\documentclass[a4paper,10pt]{article}

\usepackage{amssymb}
\usepackage{graphicx}
\usepackage{bbm}
\usepackage{mathrsfs}
\usepackage{amsmath}
\usepackage{amsthm}
\usepackage[center]{caption}
\usepackage{enumerate}
\usepackage{verbatim}
\usepackage{authblk}

\theoremstyle{plain}
\newtheorem{thm}{Theorem}[]

\newtheorem{lem}[thm]{Lemma}
\newtheorem{prop}[thm]{Proposition}

\theoremstyle{definition}

\newtheorem*{rmk}{Remark}

\newcommand{\Pb}{\mathbb{P}}
\newcommand{\Eb}{\mathbb{E}}

\newcommand{\Fg}{\mathcal{F}}

\newcommand{\hs}{\hspace{2mm}}

\newcommand{\ind}{\mathbbm{1}}

\title{Almost sure asymptotics for the\\random binary search tree}
\author{Matthew I.~Roberts \thanks{Laboratoire de Probabilit\'es et Mod\`eles Al\'eatoires, Universit\'e Paris VI, 175 rue du Chevaleret, 75013 Paris. \emph{E-mail:} matthew.roberts@upmc.fr . This work was supported by ANR MADCOF, grant ANR-08-BLAN-0220-01.}}

\begin{document}

\maketitle

\subsection*{Abstract}
We consider a (random permutation model) binary search tree with $n$ nodes and give asymptotics on the $\log\log$ scale for the height $H_n$ and saturation level $h_n$ of the tree as $n\to\infty$, both almost surely and in probability. We then consider the number $F_n$ of particles at level $H_n$ at time $n$, and show that $F_n$ is unbounded almost surely.

This is a work in progress --- we hope to give further results on the asymptotics of $F_n$.

\section{Introduction and main results}
Consider the complete rooted binary tree $\mathbb{T}$. We construct a sequence $\mathbb{T}_n$, $n=1,2,\ldots$ of subtrees of $\mathbb{T}$ recursively as follows. $\mathbb{T}_1$ consists only of the root. Given $\mathbb{T}_n$, we choose a leaf $u$ uniformly at random from the set of all leaves of $\mathbb{T}_n$ and add its two children to the tree to create $\mathbb{T}_{n+1}$. Thus $\mathbb{T}_{n+1}$ consists of $\mathbb{T}_n$ and the children $u1$, $u2$ of $u$, and contains in total $2n+1$ nodes, including $n+1$ leaves. We call this sequence of trees $(\mathbb{T}_n)_{n\geq1}$ the binary search tree.

\begin{figure}[h!]
  \centering
      \includegraphics[width=\textwidth]{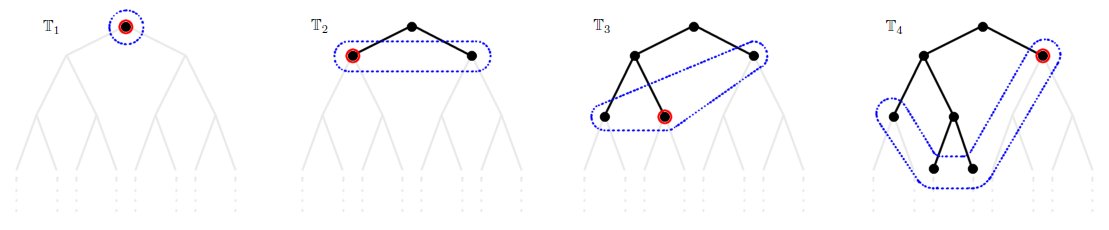}
  \caption{An example of the beginning of a binary search tree: at each stage, we choose uniformly at random from amongst the available leaves and add the children of the chosen leaf to the tree.}
\end{figure}

This model has various equivalent descriptions: for example one may construct $\mathbb{T}_n$ by successive insertions into $\mathbb{T}$ of a uniform random permutation of $\{1,\ldots,n\}$. For a more detailed explanation of this and other constructions see Reed \cite{reed:height_bst}.

One interesting quantity in this model is the height $H_n$ of the tree $\mathbb{T}_n$ --- that is, the greatest generation amongst all nodes of $\mathbb{T}_n$ (where the root is defined to have generation 0); so $H_1=0$, $H_2=1$, $H_3=2$, and $H_4$ is either 2 (with probability 1/3) or 3 (with probability 2/3). Another is the saturation level $h_n$, defined to be the greatest complete generation of $\mathbb{T}_n$ --- that is, the greatest generation $k$ such that all nodes of generation $k$ are present in $\mathbb{T}_n$ (so $h_1=0$, $h_2=1$, $h_3=1$ and $h_4$ is 1 with probability $2/3$ and 2 with probability $1/3$). These two quantities, $H_n$ and $h_n$, have been studied extensively. Pittel \cite{pittel:growing_random_binary_trees} showed that there exist constants $c$ and $\gamma$ such that $H_n/\log n\to c$ and $h_n/\log n\to \gamma$ almost surely, and gave bounds on the values of $c$ and $\gamma$. Devroye \cite{devroye:note_height} calculated $c$ exactly by showing that $H_n/\log n\to c$ in probability as $n\to\infty$; and Reed \cite{reed:height_bst} showed that for the same $c$ and another known constant $d$, $\Eb[H_n] = c\log n - d\log\log n + O(1)$.  Drmota \cite{drmota:analytic_approach_height_bst_2} and Reed \cite{reed:height_bst} also showed that $\hbox{Var} H_n = O(1)$.

Our first aim in this article is to prove the following theorem.

\begin{thm}\label{main_thm}
Let $a$ be the solution to
\[2(a-1)e^a + 1 = 0,\hs\hs a>0\]
and let
\[b:=2ae^a\]
(we get $a\approx 0.76804$ and $b\approx 3.31107$). Then
\[\frac{1}{2} = \liminf_{n\to\infty}\frac{b\log n - aH_n}{\log\log n} < \limsup_{n\to\infty}\frac{b\log n - aH_n}{\log\log n} = \frac{3}{2}\]
almost surely and
\[\frac{b\log n - aH_n}{\log\log n} \overset{\Pb}\longrightarrow \frac{3}{2} \hs \hbox{ as } \hs n\to\infty.\]
\end{thm}
Of course, $a$ and $b$ agree with the constants $c$ and $d$ mentioned above in the sense that $c=b/a$ and $d=3/2a$. By the same methods, we obtain a similar theorem concerning $h_n$.

\begin{thm}\label{h_thm}
Let $\alpha$ be the solution to
\[2(\alpha+1)e^{-\alpha} -1 = 0, \hs\hs \alpha>0\]
and let
\[\beta:=2\alpha e^{-\alpha}\]
(we get $\alpha\approx 1.6783$ and $\beta\approx 0.6266$). Then
\[\frac{1}{2} = \liminf_{n\to\infty}\frac{\alpha h_n - \beta\log n}{\log\log n} < \limsup_{n\to\infty}\frac{\alpha h_n - \beta\log n}{\log\log n} = \frac{3}{2}\]
almost surely and
\[\frac{\alpha h_n - \beta\log n}{\log\log n} \overset{\Pb}\longrightarrow \frac{3}{2} \hs \hbox{ as } \hs n\to\infty.\]
\end{thm}

\noindent
This shows in particular that the lower bound given by Pittel \cite{pittel:growing_random_binary_trees} is the correct growth rate for the saturation level $h_n$ on the $\log$ scale.

\vspace{2mm}

Other aspects of the binary search tree model also give interesting results. The article by Chauvin \emph{et al.} \cite{chauvin_et_al:mgs_profile_bst}, for example, tracks the number of leaves at certain levels of the tree, called the profile of the tree, via convergence theorems for polynomial martingales associated with the system.

We are also interested in how many leaves are present at level $H_n$ of the tree at time $n$. We call the set of particles at this level the \emph{fringe} of the tree, and call the size of the fringe $F_n$, so that $F_1 = 1$, $F_2=2$, $F_3=2$, and $F_4$ is 2 with probability $2/3$ or 4 with probability $1/3$. Note that the word ``fringe'' has been used also in a different context by, for example, Drmota \emph{et al.} \cite{drmota_et_al:shape_fringe_various_trees}. Trivially $F_n\in \{2,4,6,\ldots\}$ for all $n\geq2$, and (given that $H_n\to\infty$ almost surely, which is a simple consequence of Theorem \ref{main_thm}) $\liminf_{n\to\infty} F_n = 2$ almost surely. We are able to prove the following preliminary result.

\begin{prop}\label{F_prop}
We have
\[\limsup_{n\to\infty} F_n = \infty\]
almost surely.
\end{prop}

\noindent
Further work on the behaviour of $F_n$ in the limit as $n\to\infty$ is underway.

\vspace{3mm}

Our main tool throughout is the relationship between binary search trees and an extremely simple continuous time branching random walk, called the Yule tree. This relationship is well-known --- see Aldous \& Shields \cite{aldous_shields:diffusion_limit_random_binary_trees} and Chauvin \emph{et al.} \cite{chauvin_et_al:mgs_profile_bst}. The hard work required for Theorem \ref{main_thm} is then done for us by a remarkable result of Hu \& Shi \cite{hu_shi:minimal_pos_crit_mg_conv_BRW}. We introduce the Yule tree model in Section \ref{yule_sec} before proving Theorems \ref{main_thm} and \ref{h_thm} in Section \ref{main_thm_sec}. Finally we study $F_n$, and in particular prove Proposition \ref{F_prop}, in Section \ref{F_sec}.

\section{The Yule tree}\label{yule_sec}
Consider a branching random walk in continuous time with branching rate 1, starting with one particle at the origin, in which if a particle with position $x$ branches it is replaced by two children with position $x-1$. That is:
\begin{itemize}
\item We begin with one particle at 0;
\item All particles act independently;
\item Each particle lives for a random amount of time, exponentially distributed with parameter 1;
\item Each particle has a position $x$ which does not change throughout its lifetime;
\item At its time of death, a particle with position $x$ is replaced by two offspring with position $x-1$.
\end{itemize}
We call this process a Yule tree. Let $N(t)$ be the set of particles alive at time $t$, and for a particle $u\in N(t)$ define $X_u(t)$ to be the position of $u$ at time $t$. Let $M(t)$ denote the smallest of these positions at time $t$, and $S(t)$ the largest --- that is,
\[M(t):=\inf\{X_u(t) : u\in N(t)\}\]
and
\[S(t):=\sup\{X_u(t) : u\in N(t)\}.\]

We note that if we look at the Yule tree model only at integer times, then we have a discrete-time branching random walk. On the other hand, we have the following simple relationship between the Yule tree process and the binary search tree process.

\begin{lem}
Let $T_1=0$ and for $n\geq2$ define
\[T_n := \inf\{t> T_{n-1} : N(t) \neq N(T_{n-1})\}\]
so that the times $T_n$ are the birth times of the branching random walk. Then we may construct the Yule tree process and the binary search tree process on the same probability space, such that
\[-M(T_n) = H_n \hs \forall n\geq1\]
and
\[-S(T_n) = h_n + 1 \hs \forall n\geq1\]
almost surely.
\end{lem}

\begin{proof}
By the memoryless property of the exponential distribution, at any time $t$ the probability that a particular particle $u\in N(t)$ will be the next to branch is exactly $1/\#N(t)$. Thus, if we consider the sequence of genealogical trees produced by the Yule tree process at the times $T_j, j\geq1$, we have exactly the binary search tree process --- particles in $N(t)$ correspond to leaves in the binary search tree. Clearly the position of a particle in the Yule tree process is -1 times its height in the genealogical tree, so we may build the Yule tree process and binary search tree process on the same probability space and then $-M(T_n) = H_n$ and $-S(T_n) = h_n + 1$ for all $n\geq1$ (almost surely).
\end{proof}

We would like to study $(H_n,n\geq1)$ via knowledge of $(M(t),t\geq 0)$, and similarly for $h_n$ and $S(t)$, and hence it will be important to have control over the times $T_n$. It is well-known that $T_n$ is close to $\log n$. We give a simple martingale proof, as seen in Athreya \& Ney \cite{athreya_ney:branching_processes}.

\begin{lem}\label{log_times_lem}
There exists an almost surely finite random variable $\zeta$ such that
\[T_n - \log n\to \zeta \hs\hs \hbox{almost surely as} \hs n\to\infty;\]
and hence for any $\delta>0$ we may choose $K\in\mathbb{N}$ such that
\[\Pb\left( \limsup_{n\to\infty} |T_n - \lfloor\log n\rfloor| > K \right) < \delta\]
and
\[\limsup_{n\to\infty} \Pb( |T_n - \lfloor\log n\rfloor| > K ) < \delta.\]
\end{lem}

\begin{rmk}
One may in fact show that $\zeta$ is exponentially distributed with parameter 1.
\end{rmk}

\begin{proof}
For each $n\geq1$, let $V_n:=n(T(n)-T(n-1))$. Then the random variables $V_n$, $n\geq1$ are independent and exponentially distributed with parameter 1. Define
\[ X_n := \sum_{j=1}^n \frac{V_j - 1}{j} = T(n) - \sum_{j=1}^n j^{-1}.\]
Then $X_n$ is clearly a zero-mean martingale; and
\[\Eb[X_n^2] = \sum_{j=1}^n \frac{\hbox{Var}(V_j)}{j^2} \leq \sum_{j=1}^\infty j^{-2} <\infty\]
so by the martingale convergence theorem $X_n$ converges almost surely (and in $L^2$) to some almost surely finite limit $X$. But it is well-known that
\[\sum_{j=1}^n j^{-1} - \log n\]
converges to some finite, deterministic constant. This is enough to complete the proof of the first statement in the Lemma, and the next part is trivial: since $\zeta$ is almost surely finite, we may choose $K$ such that $\Pb(|\zeta|>K)<\delta$. For the final part, we may either use Fatou's lemma:
\begin{align*}
\limsup_{n\to\infty}\Pb(|T_n - \log n| > K+1) &\leq \Eb\left[\limsup_{n\to\infty}\ind_{\{|T_n - \log n| > K+1\}}\right]\\
&\leq \Pb\left(\limsup_{n\to\infty}|T_n - \log n| > K\right);
\end{align*}
or, for a more elementary proof, apply Chebyshev's inequality to the martingale $X_n$:
\[\Pb\left(\bigg|T_n - \sum_{j=1}^n j^{-1}\bigg|>K\right) = \Pb(|X_n|>K) \leq \frac{\Eb[X_n^2]}{K^2}.\qedhere\]
\end{proof}

We mentioned above that, if we look at the Yule tree only at integer times, we see a discrete-time branching random walk. Since discrete-time branching random walks are more widely studied than their continuous-time counterparts (in particular the theorem that we would like to apply is stated only in discrete-time), it will be helpful to know the branching distribution of the discrete model. This is a standard calculation.

\begin{lem}\label{calc_lem}
We have
\[\Eb\left[\sum_{u\in N(1)} e^{-\theta X_u(1)}\right] = \exp(2e^\theta - 1).\]
\end{lem}

\begin{proof}
Let
\[E_\theta(t) = \Eb\left[\sum_{u\in N(t)} e^{-\theta X_u(t)}\right],\]
and for $s,t\geq0$ and a particle $u\in N(t)$ define $N_u(t;s)$ to be the set of descendants of particle $u$ alive at time $t+s$: that is, $N_u(t;s):=\{v\in N(t+s) : u\leq v\}$. Then by the Markov property,
\begin{align*}
E_\theta(t+s) &= \Eb\left[\sum_{u\in N(t+s)} e^{-\theta X_u(t+s)}\right]\\
&= \Eb\left[\sum_{u\in N(t)}e^{-\theta X_u(t)}\sum_{v\in N_u(t;s)} e^{-\theta(X_v(t+s) - X_v(t))}\right]\\
&= \Eb\left[\sum_{u\in N(t)}e^{-\theta X_u(t)}\Eb\left[\left.\sum_{v\in N_u(t;s)} e^{-\theta(X_v(t+s) - X_v(t))}\right|\Fg_t\right]\right]\\
&= \Eb\left[\sum_{u\in N(t)}e^{-\theta X_u(t)}E_\theta(s)\right]\\ 
&= E_\theta(t)E_\theta(s).
\end{align*}
We deduce that for $s,t>0$,
\[\frac{E_\theta(t+s) - E_\theta(t)}{s} = E_\theta(t)\left(\frac{E_\theta(s)-1}{s}\right)\]
and
\[\frac{E_\theta(t-s) - E_\theta(t)}{-s} = E_\theta(t-s)\left(\frac{E_\theta(s)-1}{s}\right).\]
It is easily checked that $E_\theta(t)$ is continuous in $t$, and hence if $E'_\theta(0+)$ exists then by the above we have that $E_\theta(t)$ is continuously differentiable and for all $t>0$
\[E'_\theta(t) = E_\theta(t)E'_\theta(0+).\]
Since $E_\theta(0)=1$ this entails that
\[E_\theta(t) = \exp(E'_\theta(0+)t).\]

Now, for small $t$,
\begin{align*}
E_\theta(t) &= \Pb(\hbox{first split after } t) + 2e^\theta\Pb(\hbox{first split before } t) + o(t)\\
&= 1-t + 2te^\theta + o(t)
\end{align*}
so that $E'(0+) = 2e^\theta - 1$, and hence $E_\theta(t) = \exp((2e^\theta-1)t)$. Taking $t=1$ completes the proof.
\end{proof}

These simple properties of the Yule tree will allow us to prove our main theorem.

\section{Proof of Theorems \ref{main_thm} and \ref{h_thm}}\label{main_thm_sec}
We would like to apply the following theorem of Hu and Shi \cite{hu_shi:minimal_pos_crit_mg_conv_BRW}. This result was proved for a large class of branching random walks; our particular simple case (when recentred) trivially satisfies the assumptions in \cite{hu_shi:minimal_pos_crit_mg_conv_BRW}, and so we omit those assumptions here.

\begin{thm}[Hu, Shi \cite{hu_shi:minimal_pos_crit_mg_conv_BRW}]\label{hu_shi_thm}
Define
\[\psi(\theta) := \Eb\left[\sum_{u\in N(1)} e^{-\theta X_u(1)}\right].\]
If $\theta^\ast$ satisfies
\[\frac{\theta^\ast\psi'(\theta^\ast)}{\psi(\theta^\ast)} = \log \psi(\theta^\ast), \hs\hs \theta^\ast>0,\]
then
\[\frac{1}{2} = \liminf_{n\to\infty}\frac{\theta^\ast M(n) + n\log\psi(\theta^\ast)}{\log n} < \limsup_{n\to\infty}\frac{\theta^\ast M(n) + n\log\psi(\theta^\ast)}{\log n} = \frac{3}{2}\]
and
\[\frac{\theta^\ast M(n) + n\log\psi(\theta^\ast)}{\log n} \overset{\Pb}\longrightarrow \frac{3}{2} \hs \hbox{ as } \hs n\to\infty.\]
\end{thm}

In view of this result, our method of proof for Theorems \ref{main_thm} and \ref{h_thm} is unsurprising: we know that the times $T_n$ are near $\log n$ for large $n$, and we may use the monotonicity of $H_n$ and $h_n$ --- together with the flexibility offered by the $\log\log$ scale --- to ensure that nothing else can go wrong. It may be possible to extend this method of proof to cover more general trees, where the same monotonicity property does not necessarily hold, via a Borel-Cantelli argument. This would only introduce unneccessary complications in our case.

\begin{proof}[Proof of Theorem \ref{main_thm}]
We show first the statement involving the limsup; the proofs of the other statements are almost identical.

It is immediate from Lemma \ref{calc_lem} that $a$ in Theorem \ref{main_thm} corresponds to $\theta^\ast$ in Theorem \ref{hu_shi_thm}, and that $b$ corresponds to $\log\psi(\theta^\ast)$. Fix $\delta > 0$. Choose $K\in\mathbb{N}$ such that
\[\Pb(\limsup_{n\to\infty} |T_n - \lfloor\log n\rfloor| > K ) <\delta\]
--- this is possible by Lemma \ref{log_times_lem}. For each $n\geq1$, let $j_n = \lfloor\log n\rfloor - K$. We use the abbreviation ``i.o.'' to mean ``infinitely often'' --- that is, for a sequence of measurable sets $U_n$, $\{U_n \hs\hbox{i.o.}\}$ represents the event $\limsup_{n\to\infty} U_n$. For any $\varepsilon>0$, using the fact that $M(t)$ is non-increasing,
\begin{align*}
&\Pb(aM(T_n) + b\log n > (3/2+\varepsilon)\log\log n \hs\hbox{i.o.})\\
&\leq \Pb(\{aM(T_n) + b\log n > (3/2+\varepsilon)\log\log n, |T_n - \lfloor\log n\rfloor|\leq K\} \hs\hbox{i.o.})\\
&\hspace{20mm} + \Pb(|T_n - \lfloor\log n\rfloor| > K \hs\hbox{i.o.})\\
&< \Pb\bigg( aM(j_n) > -bj_n + (3/2+\varepsilon)\log j_n + (bj_n - b\log n)\\
&\hspace{50mm} + (3/2+\varepsilon)(\log\log n - \log j_n) \hs\hbox{i.o.}\bigg) + \delta\\
&\leq \Pb( aM(j_n) > -bj_n + (3/2+\varepsilon/2)\log j_n \hs\hbox{i.o.}) + \delta\\
&\leq \delta
\end{align*}
by Theorem \ref{hu_shi_thm}. Taking a union over $\varepsilon>0$ tells us that
\[\Pb\left(\limsup_{n\to\infty} \frac{aM(T_n) + b\log n}{\log\log n} > \frac{3}{2}\right)\leq\delta;\]
but since $\delta>0$ was arbitrary we deduce that
\[\Pb\left(\limsup_{n\to\infty} \frac{aM(T_n) + b\log n}{\log\log n} > \frac{3}{2}\right)=0.\]
This completes the proof of the upper bound, since $H_n = -M(T_n)$. The proof of the lower bound is similar. We let $i_n = \lfloor\log n\rfloor + K$ and use the abbreviation ``ev.'' to mean ``eventually'' (that is, for all large $n$; so $\{U_n \hs\hbox{ev.}\}$ represents the event $\liminf_{n\to\infty}U_n$). For any $\varepsilon\in(0,3/2)$,
\begin{align*}
&\Pb(aM(T_n) + b\log n < (3/2-\varepsilon)\log\log n \hs\hbox{ev.})\\
&\leq \Pb(\{aM(T_n) + b\log n < (3/2-\varepsilon)\log\log n, |T_n - \lfloor\log n\rfloor|\leq K\} \hs\hbox{ev.})\\
&\hspace{20mm} + \Pb(|T_n - \lfloor\log n\rfloor| > K \hs\hbox{i.o.})\\
&< \Pb\bigg( aM(i_n) < -bi_n + (3/2-\varepsilon)\log i_n + (bi_n - b\log n)\\
&\hspace{50mm} + (3/2-\varepsilon)(\log\log n - \log i_n) \hs\hbox{ev.}\bigg) + \delta\\
&\leq \Pb( aM(i_n) < -bi_n + (3/2-\varepsilon/2)\log i_n \hs\hbox{ev.}) + \delta\\
&\leq \delta
\end{align*}
by Theorem \ref{hu_shi_thm}. As with the upper bound, taking a union over $\varepsilon>0$, and then letting $\delta\to0$, tells us that
\[\Pb\left(\limsup_{n\to\infty} \frac{aM(T_n) + b\log n}{\log\log n} < \frac{3}{2}\right)=0\]
and hence combining with the upper bound we obtain
\[\limsup_{n\to\infty}\frac{b\log n - aH_n}{\log\log n} = \frac{3}{2}\]
almost surely. The proof of the statement involving the liminf is almost identical, and we omit it for the sake of brevity. The convergence in probability is also similar: one considers for example that
\begin{align*}
&\limsup_{n\to\infty}\Pb\left(aM(T_n)+b\log n > (3/2+\varepsilon)\log\log n\right)\\
&\leq \limsup_{n\to\infty}\Pb\left(a M(T_n) + b\log n > (3/2+\varepsilon)\log\log n, \hs |T_n-\lfloor \log n\rfloor|\leq K\right)\\
&\hspace{20mm} + \limsup_{n\to\infty}\Pb\left(|T_n-\lfloor \log n\rfloor|>K\right)\\
&< \limsup_{n\to\infty}\Pb\left(a M(T_n) + b\log n > (3/2+\varepsilon)\log\log n, \hs |T_n-\lfloor \log n\rfloor|\leq K\right) + \delta
\end{align*}
and uses the statement about convergence in probability in Theorem \ref{hu_shi_thm} to show that the probability in the last line above converges to zero for any $\varepsilon>0$. Then since $\delta>0$ was arbitrary we must have
\[\limsup_{n\to\infty}\Pb\left(aM(T_n)+b\log n > (3/2+\varepsilon)\log\log n\right)=0.\]
The lower bound is, again, similar.
\end{proof}

\begin{proof}[Proof of Theorem \ref{h_thm}]
Consider a slightly altered Yule tree model, where each particle gives birth to two children whose position is that of their parent \emph{plus} 1, instead of minus 1. If we couple this model with the usual Yule tree model in the obvious way, then clearly the minimal position of a particle in the altered model is equal to $-1$ times the maximal position in the usual model. Thus if we let $\hat M(t)$ be the minimal position in the altered model, it suffices to show that
\[\frac{1}{2} = \liminf_{n\to\infty}\frac{\alpha \hat M(T_n) - \beta\log n}{\log\log n} < \limsup_{n\to\infty}\frac{\alpha \hat M(T_n) - \beta\log n}{\log\log n} = \frac{3}{2}\]
and
\[\frac{\alpha \hat M(T_n) - \beta\log n}{\log\log n} \overset{\Pb}\longrightarrow \frac{3}{2} \hs \hbox{ as } \hs n\to\infty.\]

Lemma \ref{calc_lem} (substituting $\hat \theta := -\theta$, say) tells us that for the altered model, $\alpha$ in Theorem \ref{h_thm} corresponds to $\theta^\ast$ in Theorem \ref{hu_shi_thm}, and that $-\beta$ corresponds to $\log\psi(\theta^\ast)$. The rest of the proof proceeds exactly as in the proof of Theorem \ref{main_thm}.
\end{proof}

\section{The size of the fringe, $F_n$}\label{F_sec}
We are now interested in the size of the fringe of the tree: how many leaves lie at level $H_n$ at time $n$. Recall that we called this quantity $F_n$.

\begin{figure}[h!]
  \centering
      \includegraphics[width=\textwidth]{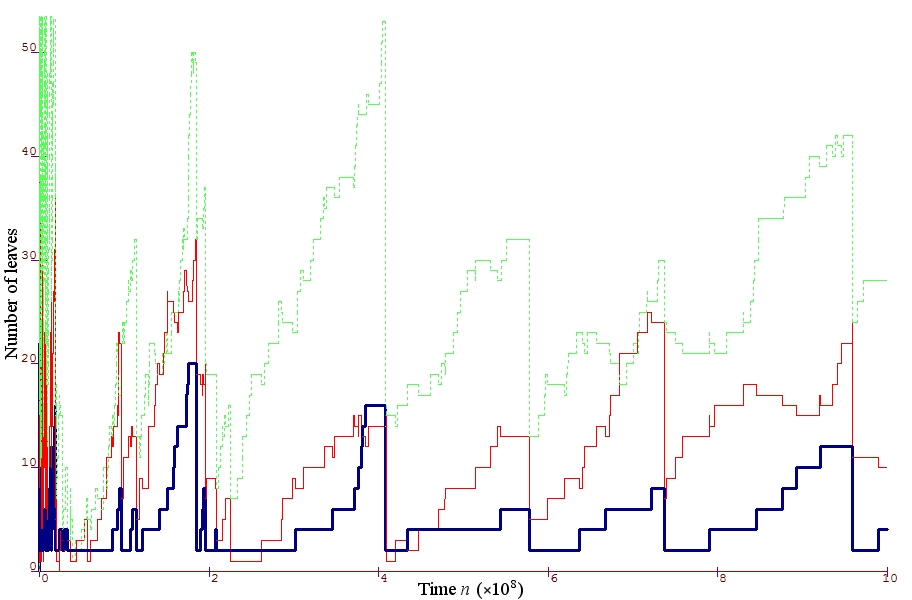}
  \caption{The top three levels of a binary search tree run for $10^9$ steps. The thick blue line shows the size of the fringe, $F_n$, which is the number of leaves at level $H_n$; the thin red line shows the number of leaves at level $H_n - 1$; and the dashed green line shows the number of leaves at level $H_n - 2$.}\label{picture}
\end{figure}

We will show that $F_n$ is unbounded almost surely, but first we need a short lemma. For this lemma we consider again the Yule tree model, and call the set of particles with position $M(t)$ the \emph{frontier} of the Yule tree at time $t$ --- recall that this is the set of particles with minimal position at time $t$, so as we saw earlier the frontier of the Yule tree corresponds to the fringe of the binary search tree.  Define $\tilde F_t$ to be the number of particles at the frontier at time $t$,
\[\tilde F_t := \#\{u\in N(t) : X_u(t) = M(t)\}.\]

\begin{lem}\label{frontier_lem}
If $M(t) < -\lfloor\log_2(2k)\rfloor$ and $\tilde F_t = 2k$, then there is at least one particle that is not at the frontier at time $t$, but which is within distance $\lfloor\log_2(2k)\rfloor$ of the frontier --- that is, its position is in the interval $[M(t)+1, M(t)+\lfloor\log_2(2k)\rfloor]$.
\end{lem}

\begin{proof}
Clearly at some time before $t$ there was a particle which had position $M(t) + \lfloor\log_2(2k)\rfloor$; and hence at some time there were at least 2 particles with this position, since particles (except the root) arrive in pairs. At time $t$, either these particles have at least one descendant not at the frontier, in which case we are done (as particles cannot move in the positive direction); or all their descendants are at the frontier. So, for a contradiction, suppose that all their descendants \emph{are} at the frontier at time $t$. Then there must be $2\times 2^{\lfloor\log_2(2k)\rfloor}$ particles at the frontier (since a movement of distance 1 yields 2 new particles, and hence a movement of distance $\lfloor\log_2(2k)\rfloor$ yields $2^{\lfloor\log_2(2k)\rfloor}$ new particles; and this holds for each of the two initial particles). But
\[2\times 2^{\lfloor\log_2(2k)\rfloor} = 2^{\lfloor\log_2(2k)\rfloor + 1} > 2^{\log_2(2k)} = 2k\]
so there are strictly more than $2k$ particles at the frontier. This is a contradiction --- there are exactly $2k$ particles at the frontier, by assumption --- and hence our claim holds.
\end{proof}

We now prove Proposition \ref{F_prop}, which we recall says that $\limsup_{n\to\infty}F_n=\infty$ almost surely.

\begin{proof}[Proof of Proposition \ref{F_prop}]
Again consider the continuous time Yule tree. By the relationship between the Yule tree and the binary search tree seen in Section \ref{yule_sec}, $\tilde F_t$ and $F_n$ have the same paths up to a time change, and hence it suffices to show that $\limsup\tilde F_t = \infty$ almost surely.

The idea is as follows: suppose we have $2k$ particles at the frontier. By Lemma \ref{frontier_lem}, there is a particle close to the frontier; and this particle has probability greater than some strictly positive constant of having 2 of its descendants make it to the frontier before the $2k$ already there branch. So if we have 2k particles infinitely often, then we have $2k+2$ particles infinitely often. We make this argument rigorous below.

For any $t>0$ and $k\in\mathbb{N}$, define
\[\tau^{(2k)}_1 := \inf\{s>0 : M(s) < -\lfloor\log_2(2k)\rfloor \hbox{ and } \tilde F_s = 2k\}\]
and for each $j\geq1$
\[\sigma^{(2k)}_{j} := \inf\{s>\tau^{(2k)}_j : \tilde F_s \neq 2k\}\]
and
\[\tau^{(2k)}_{j+1} := \inf\{s>\sigma^{(2k)}_j : \tilde F_s = 2k\}.\]
Then $\tau^{(2k)}_j$ is the $j$th time that we have $2k$ particles at the frontier and at least distance $\log_2(2k)$ from the origin. We show, by induction on $k$, that for any $k\in\mathbb{N}$
\begin{equation}\label{induction}
\tau^{(2k)}_j<\infty \hs\hbox{ almost surely,}\hs \hbox{for all}\hs j\in\mathbb{N}.
\end{equation}
Trivially, since $M(t)\to-\infty$ almost surely (which is true since $H_n\to\infty$ almost surely), we have $\tau^{(2)}_j<\infty$ almost surely for all $j\in\mathbb{N}$ and so (\ref{induction}) holds for $k=1$. Suppose now (\ref{induction}) holds for some $k\geq1$.

By Lemma \ref{frontier_lem}, for any $j$, at time $\tau^{(2k)}_j$ there is at least 1 particle that is not at the frontier but is within distance $\lfloor\log_2(2k)\rfloor$ of the frontier. Let $A^{(k)}_j$ be the event that the descendants of this particle reach level $M(\tau^{(2k)}_j)$ before any of the $2k$ particles already at that level branch. Then the events $A^{(k)}_1, A^{(k)}_2, A^{(k)}_3,\ldots$ are independent by the strong Markov property. Also, since all particles branch at rate 1, for each $j$ the probability of $A^{(k)}_j$ is certainly at least the probability that the sum of $\lfloor\log_2(2k)\rfloor$ independent, rate 1 exponential random variables is less that the minimum of $2k$ independent, rate 1 exponential random variables. This is some strictly positive number, $\gamma_k$ say.

Now, at time $\tau^{(2k)}_j$ --- which is finite for each $j$, by our induction hypothesis --- there are $2k$ particles at the frontier. One of two things can happen: either two more particles join them and we reach $2k+2$ particles at the frontier, or one of the $2k$ branches before this happens and we have a new frontier with 2 particles. Call the first event, that two more particles reach the frontier before any of the $2k$ already there branch, $B^{(k)}_j$. Then $A^{(k)}_j\subseteq B^{(k)}_j$ since the event that some pair makes it to the frontier before the $2k$ branch contains the event that descendants of our particular particle make it to the frontier before the $2k$ branch. Thus
\begin{align*}
\Pb\left(\limsup_{m\to\infty} B^{(k)}_m\right) &\geq \Pb\left(\limsup_{m\to\infty} A^{(k)}_m\right) =\Pb\left(\bigcap_{n\geq1}\bigcup_{m\geq n} A_m^{(k)}\right)\\
&=\lim_{n\to\infty} \Pb\left(\bigcup_{m\geq n} A_m^{(k)}\right) =\lim_{n\to\infty}\lim_{N\to\infty}\Pb\left(\bigcup_{m=n}^N A_m^{(k)}\right)\\
&\geq \lim_{n\to\infty}\lim_{N\to\infty}\left(1-(1-\gamma_k)^{N-n+1}\right) =1.
\end{align*}
But the event $\limsup_{m\to\infty} B^{(k)}_m$ is exactly the event that we have $2k+2$ particles at the frontier infinitely often --- and thus (using again that $M(t)\to -\infty$ almost surely) we have that $\tau^{(2k+2)}_j$ is finite almost surely for all $j$. Hence by induction we have proved that (\ref{induction}) holds for each $k$. Our result follows.
\end{proof}

\subsection*{Acknowledgements}
Many thanks to Julien Berestycki and Brigitte Chauvin for their ideas and their help with this project.

\bibliographystyle{plain}

\begin{thebibliography}{1}

\bibitem{aldous_shields:diffusion_limit_random_binary_trees}
D.~Aldous and P.~Shields.
\newblock  A diffusion limit for a class of randomly-growing binary trees.
\newblock {\em Probab. Theory Related Fields}, 79:509--542, 1988.

\bibitem{athreya_ney:branching_processes}
K.B. Athreya and P.E. Ney.
\newblock {\em Branching Processes}.
\newblock Springer-Verlag, New York, 1972.

\bibitem{chauvin_et_al:mgs_profile_bst}
B.~Chauvin, T.~Klein, J-F. Marckert, and A.~Rouault.
\newblock Martingales and profile of binary search trees.
\newblock {\em Electron. J. Probab.}, 10(12):420--435, 2005.

\bibitem{devroye:note_height}
L.~Devroye.
\newblock A note on the height of binary search trees.
\newblock {\em J. Assoc. Comput. Mach.}, 33(3):489--498, 1986.

\bibitem{drmota:analytic_approach_height_bst_2}
M.~Drmota.
\newblock An analytic approach to the height of binary search trees {II}.
\newblock {\em J. ACM}, 50(3):333--374, 2003.

\bibitem{drmota_et_al:shape_fringe_various_trees}
M.~Drmota, B.~Gittenberger, A.~Panholzer, H.~Prodinger, and M.D. Ward.
\newblock On the shape of the fringe of various types of random trees.
\newblock {\em Math. Methods Appl. Sci.}, 32(10):1207--1245, 2009.

\bibitem{hu_shi:minimal_pos_crit_mg_conv_BRW}
Y.~Hu and Z.~Shi.
\newblock Minimal position and critical martingale convergence in branching
  random walks, and directed polymers on disordered trees.
\newblock {\em Ann. Probab.}, 37(2):742--789, 2009.

\bibitem{pittel:growing_random_binary_trees}
B.~Pittel.
\newblock On growing random binary trees.
\newblock {\em J. Math. Anal. Appl.}, 103:461--480, 1984.

\bibitem{reed:height_bst}
B.~Reed.
\newblock The height of a random binary search tree.
\newblock {\em J. ACM}, 50(3):306--332, 2003.

\end{thebibliography}
\def\cprime{$'$}

\end{document}